\documentclass[a4paper]{jpconf} 
\usepackage{iopams,amsgen,amsfonts,amssymb,amsbsy}
\input cyracc.def

\newtheorem{Lem}{Lemma}[section] 
\newtheorem{Prop}[Lem]{Proposition} 
\newtheorem{Thm}[Lem]{Theorem} 
\newtheorem{Cor}[Lem]{Corollary} 
\newtheorem{Def}[Lem]{Definition} 
\newtheorem{Rem}[Lem]{Remark} 
 
\newtheorem{Prbl}[Lem]{Problem} 
 
\newtheorem{Alg}[Lem]{Algorithm} 
\newenvironment{proof}{{\sc Proof.}}{$\Box$}

\begin{document} 

\title{Self-adjoint symmetry operators connected with the magnetic Heisenberg ring} 
\author{Bernd Fiedler}
\address{Eichelbaumstr. 13, D-04249 Leipzig, Germany. URL: http://www.fiemath.de/}  
\ead{bfiedler@fiemath.de}  

\begin{abstract}
In \cite{fie07c} we defined symmetry classes, commutation symmetries and symmetry operators in the Hilbert space $\mathcal{H}$ of the {\rm 1D} spin-{\rm 1/2} Heisenberg magnetic ring with $N$ sites and investigated them by means of tools from the representation theory of symmetric groups $\mathcal{S}_N$ such as decompositions of ideals of the group ring $\mathbb{C}[\mathcal{S}_N]$, idempotents of $\mathbb{C}[\mathcal{S}_N]$, discrete Fourier transforms of $\mathcal{S}_N$, Littlewood-Richardson products. 

In the present paper we consider symmetry operators $a\in\mathbb{C}[\mathcal{S}_N]$ on $\mathcal{H}$ which fulfil\\
\parbox{12cm}{\[
\overline{a}^{\ast} = a,
\]}\hfill(S)\\
where $\overline{a}$ denotes the complex conjugate of $a$ and the element $a^{\ast}\in\mathbb{C}[\mathcal{S}_N]$ of $a = \sum_p a_p p$ is defined by $a^{\ast} = \sum_p a_p p^{-1}$. Such symmetry operators are {\it self-adjoint} because of the relation $\langle au | v\rangle = \langle u |\overline{a}^{\ast}v\rangle$ which holds true for all $a\in\mathbb{C}[\mathcal{S}_N]$ and $u,v\in\mathcal{H}$. They yield consequently {\it observables} of the quantum mechanical Heisenberg model. We prove the following results:

{\bf (i)} Let $G\subseteq\mathcal{S}_N$ be an arbitrary subgroup and $\chi$ be an irreducible character of $G$ then $\chi := \frac{\chi(\mathrm{id})}{|G|}\sum_{p\in G} \chi(p)\,p$ is an idempotent of $\mathbb{C}[\mathcal{S}_N]$ which has property (S). This leads to a big manifold of observables. In particular every commutation symmetry belonging to $G$ yields a 1-dimensional character of $G$ from which we can build one of the idempotents $\chi$ described here.

{\bf (ii)} Let $\mathcal{R}\subseteq\mathbb{C}[\mathcal{S}_N]$ be a minimal right ideal of $\mathbb{C}[\mathcal{S}_N]$. Then the set of all generating (primitive) idempotents of $\mathcal{R}$ contains one and only one idempotent $e$ which satisfies $\overline{e}^{\ast} = e$. This is a result by H. Weyl \cite{weyl4} for which we give a new proof.

{\bf (iii)} Every idempotent $e$ with property (S) can be decomposed into primitive idempotents $e = f_1 + \ldots + f_k$ which have also property (S) and satisfy $f_i\cdot f_j = 0$ if $i\not= j$. We give a computer algorithm for the calculation of such decompositions.

In big group rings $\mathbb{C}[\mathcal{S}_N]$ computer calculations are only possible by the use of a diskrete Fourier transform $D: \mathbb{C}[\mathcal{S}_N]\rightarrow\bigoplus_{\lambda} \mathbb{C}^{d_{\lambda}\times d_{\lambda}}$. We present two algorithms which allow the calculation of the matrix $B = D_{\lambda}(a^{\ast})$ from a matrix $A = D_{\lambda}(a)$ of an $a\in\mathbb{C}[\mathcal{S}_N]$ without to determine the inverse Fourier transform of $A$. These algorithms use as precomputed data permutation sets $\mathcal{P}_{\lambda}$ which define bases $\{ p\cdot z_{\lambda} | p\in\mathcal{P}_{\lambda}\}$ of the minimal two-sided ideals $\mathcal{Z}_{\lambda}$ belonging to the partitions $\lambda\vdash N$. ($z_{\lambda}$ is the central-primitive generating idempotent of $\mathcal{Z}_{\lambda}$). The sets $\mathcal{P}_{\lambda}$ can be used for any discrete Fourier transform $D$.

In our investigations we use computer calculations by means of the {\sf Mathematica} packages {\sf PERMS} and {\sf HRing}.
\end{abstract}

\section{The Heisenberg model, symmetry operators, symmetry classes, commutation symmetries} \label{sec2}
We summarize essential concepts of the one-dimensional (1D) spin-{1/2} Heisenberg model of a magnetic ring (see e.g. \cite{KarMu,KarHuMu,KarHuMu2} or \cite{fie07c}).

We denote by $\widehat{N}$ the set $\widehat{N}:=\{1,\ldots,N\}$ of the integers $1, 2,\ldots,N$ and by $\widehat{K}^{\widehat{N}}$ the set of all functions $\sigma : \widehat{N}\rightarrow\widehat{K}$.
\begin{Def}We assign to every function $\sigma\in\widehat{K}^{\widehat{N}}$ the sequence $|\sigma\rangle := |\sigma(1),\sigma(2),\ldots,\sigma(N)\rangle$ of its values over the set $\widehat{N}$. Then the Hilbert space of a ring model with $N$ nodes  and a spin-alphabet of $K$ letters {\rm (}see {\rm\cite{jak2lu2c})} is the set of all formal complex linear combinations
$\mathcal{H} := \mathcal{L}_{\mathbb{C}}\left\{\;|\sigma\rangle\;|\;\sigma\in\widehat{K}^{\widehat{N}}\;\right\}$ of the $|\sigma\rangle$ in which the set $\mathcal{B} := \left\{\;|\sigma\rangle\;|\;\sigma\in\widehat{K}^{\widehat{N}}\;\right\}$ of all $|\sigma\rangle$ is considered a set of linearly independent elements. We equip $\mathcal{H}$ with the scalar product
\begin{equation}
\langle\sigma |\sigma'\rangle := {\delta}_{\sigma ,\sigma'} \;\;\;\mbox{and}\;\;\;
\langle u | v\rangle := \sum_{\sigma\in\widehat{K}^{\widehat{N}}} u_{\sigma} \overline{v}_{\sigma}
\;\;\;\mbox{for}\;\;\;
u = \sum_{\sigma\in\widehat{K}^{\widehat{N}}} u_{\sigma} |\sigma\rangle
\;,\;
v = \sum_{\sigma\in\widehat{K}^{\widehat{N}}} v_{\sigma} |\sigma\rangle
\in\mathcal{H}\,.
\end{equation}
\end{Def}
Obviously, $\mathcal{H}$ is an Hilbert space of dimension
$\dim\mathcal{H} = K^N$, in which $\mathcal{B}$ is an orthonormal basis. The original 1D spin-$\frac{1}{2}$ Heisenberg ring arises for $K = 2$, where
$\sigma(k) = 2$ represents an {\it up spin} $\sigma(k) = \uparrow$ and $\sigma(k) = 1$ a {\it down spin} $\sigma(k) = \downarrow$ at site $k$.
According to {\rm\cite{KarMu,KarHuMu}}, the Hamiltonians $H_F$ {\rm (}$H_A${\rm )} of a {\rm 1D} spin-$\frac{1}{2}$ Heisenberg ferromagnet {\rm (}antiferromagnet{\rm )} of $N$ sites with periodic boundary conditions $S_{N+1}^{\alpha} := S_1^{\alpha}$, $\alpha\in\{+ , - , z\}$, are defined by
\begin{equation}
H_F := - J \sum_{k = 1}^N \left[ \frac{1}{2}\left(S_k^{+}S_{k+1}^{-} + S_k^{-}S_{k+1}^{+}\right) + S_k^z S_{k+1}^z \right]\,,\hspace*{0.5cm}H_A := - H_F\,,\hspace*{0.5cm}J = \mathrm{const.} > 0
\end{equation}
where $S_k^{+}$, $S_k^{-}$ and $S_k^z$ are spin flip operators.

Symmetry classes classes in $\mathcal{H}$ can be defined in analogy to the definition of symmetry classes of tensors (see \cite{weyl4,boerner2,weyl1,fie18,fie07c}). First we define symmetry operators.
\begin{Def}
Let $\mathcal{H}$ be the Hilbert space of a ring model of $N$ sites and $K$ letters, $\mathcal{S}_N$ be the symmetric group of the permutations of $1, 2,\ldots,N$ and $\mathbb{C}[\mathcal{S}_N]$ be the complex group ring of $\mathcal{S}_N$.
\begin{itemize}
\item[\rm (i)]{Following {\rm\cite{jak2lu2,jak2lu2b}} we define the action of a permutation $p\in\mathcal{S}_N$ on a basis vector $|\sigma\rangle\in\mathcal{H}$ by $p |\sigma\rangle := |\sigma\circ p^{-1}\rangle$.}
\item[\rm (ii)]{Every group ring element $a = \sum_p a_p p\in\mathbb{C}[\mathcal{S}_N]$ acts as so-called symmetry operator on the vectors $w = \sum_{\sigma} w_{\sigma} |\sigma\rangle\in\mathcal{H}$ by
\begin{equation}
aw :=
\sum_{p\in\mathcal{S}_N} \sum_{\sigma\in{\widehat{K}}^{\widehat{N}}} a_p w_{\sigma} p |\sigma\rangle =
\sum_{p\in\mathcal{S}_N} \sum_{\sigma\in{\widehat{K}}^{\widehat{N}}} a_p w_{\sigma} |\sigma\circ p^{-1}\rangle
\end{equation}
}
\end{itemize}
\end{Def}
A symmetry operator is a linear mapping $a:\mathcal{H}\rightarrow\mathcal{H}$ which possesses the following properties:
\begin{Prop}[\cite{fie07c}]
\begin{itemize}
\item[\rm (i)]{It holds $p(q|\sigma\rangle) = (p\circ q)|\sigma\rangle$ for all $p,q\in\mathcal{S}_N$, where ''$\circ$'' is defined by $(p\circ q)(i) := p(q(i))$.}
\item[\rm (ii)]{It holds $a(bw) = (a\cdot b)w$ for all $a,b\in\mathbb{C}[\mathcal{S}_N]$ and $w\in\mathcal{H}$, where ''$\cdot$'' denotes the multiplication $a\cdot b := \sum_p \sum_q a_p b_q p\circ q$ of group ring elements.}
\item[\rm (iii)]{It holds $\langle au | v\rangle = \langle u |\overline{a}^{\ast}v\rangle$ for all $a\in\mathbb{C}[\mathcal{S}_N]$ and $u,v\in\mathcal{H}$, where $\overline{a}$ denotes the complex conjugate of $a$ and the element $a^{\ast}\in\mathbb{C}[\mathcal{S}_N]$ of $a = \sum_p a_p p$ is defined by $a^{\ast} = \sum_p a_p p^{-1}$.}
\end{itemize}
\end{Prop}
Now we define symmetry classes in $\mathcal{H}$.
\begin{Def}
Let $\mathcal{R}\subseteq\mathbb{C}[\mathcal{S}_N]$ be a right ideal of
$\mathbb{C}[\mathcal{S}_N]$. Then
\begin{equation}
\mathcal{H}_{\mathcal{R}} := \{ au\in\mathcal{H}\;|\;a\in\mathcal{R}\,,\,
u\in\mathcal{H}\}
\end{equation}
is called the symmetry class of $\mathcal{H}$ defined by $\mathcal{R}$.
\end{Def}
As for tensors one can prove (see e.g. \cite[Chap.~V, {\S}4]{boerner2} or \cite[p.~115]{fie16}).
\begin{Prop}
Let $e\in\mathbb{C}[\mathcal{S}_N]$ be a generating idempotent of a right ideal $\mathcal{R}\subseteq\mathbb{C}[\mathcal{S}_N]$, i.e.
$\mathcal{R} = e\cdot\mathbb{C}[\mathcal{S}_N]$. Then a $u\in\mathcal{H}$ is in $\mathcal{H}_{\mathcal{R}}$ iff $eu = u$.
\end{Prop}
\begin{Cor}
If $e, f$ are generating idempotents of a right ideal $\mathcal{R}\subseteq\mathbb{C}[\mathcal{S}_N]$, then $u\in\mathcal{H}_{\mathcal{R}}$ iff $eu = fu = u$.
\end{Cor}
\begin{Def}
We denote by $\mathcal{J}_0$ the set $\mathcal{J}_0 := \{a\in\mathbb{C}[\mathcal{S}_N]\;|\;au = 0\; \forall\,u\in\mathcal{H}\}$.
\end{Def}
If $N > K$ then $\mathcal{J}_0\not=\{0\}$, because then the idempotent
$e := \frac{1}{N!} \sum_p \mathrm{sign}(p) p$ satisfies $e|\sigma\rangle = 0$ for all $|\sigma\rangle\in\mathcal{B}$.
\begin{Prop}[{\cite[p.~116]{fie16}}]
$\mathcal{J}_0$ is a two-sided ideal of $\mathbb{C}[\mathcal{S}_N]$.
\end{Prop}
Every two-sided ideal of $\mathbb{C}[\mathcal{S}_N]$ has one and only one generating idempotent which is central. Let $f_0$ be the generating idempotent of $\mathcal{J}_0$. Then $f := id - f_0$ is also a central idempotent which is orthogonal to $f_0$, i.e. $f\cdot f_0 = f_0\cdot f = 0$, and which generates a two-sided ideal $\mathcal{J} := f\cdot\mathbb{C}[\mathcal{S}_N]$ that fulfills $\mathbb{C}[\mathcal{S}_N] = \mathcal{J}\oplus \mathcal{J}_0$. $\mathcal{J}$ contains all those symmetry operators $a$ for which $\ker a\subset\mathcal{H}$. These are the symmetry operators in which we are interestet. (Compare \cite[p.~116]{fie16}.)

A special type of symmetries are symmetries of elements $u\in\mathcal{H}$ with respect to commutations of the nodes of the ring.
\begin{Def}
Let $C\subseteq\mathcal{S}_N$ be a subgroup of $\mathcal{S}_N$ and
$\epsilon: C\rightarrow\mathcal{S}^1$ be a homomorphism of $C$ onto a finite subgroup in the group
$\mathcal{S}^1 = \{z\in\mathbb{C}\;|\;|z| = 1\}$ of complex units. We say that $u\in\mathcal{H}$ possesses the commutation symmetry $(C,\epsilon)$ if
$cu = \epsilon(c)u$ for all $c\in C$.
\end{Def}
\begin{Prop}[{\cite[p.~115]{fie16}}] \label{prop2.14}%
Let $(C,\epsilon)$ be a commutation symmetry.
\begin{itemize}
\item[\rm (i)]{Then the group ring element
$\epsilon := \frac{1}{|C|}\,\sum_{c\in C}\epsilon(c)c$
is an idempotent of $\mathbb{C}[C]\subseteq\mathbb{C}[\mathcal{S}_N]$.}
\item[\rm (ii)]{A $u\in\mathcal{H}$ has the symmetry $(C,\epsilon)$ iff
${\epsilon}^{\ast}u = u$.}
\end{itemize}
\end{Prop}
Consequently, a commutation symmetry $(C, \epsilon)$ defines the symmetry class $\mathcal{H}_{\mathcal{R}}$ of the right ideal $\mathcal{R} = {\epsilon}^{\ast}\cdot\mathbb{C}[\mathcal{S}_N]$.

\section{Self-adjoint symmetry operators} \label{sec3}
In the present paper we are interested in symmetry operators $a\in\mathbb{C}[\mathcal{S}_N]$ which are self-adjoint with respect to the scalar product $\langle\cdot , \cdot\rangle$ of $\mathcal{H}$. Such symmetry operators can lead to {\it observables} of the ring model. First we find
\begin{Prop} \label{prop2.1}
Two symmetry operators $a, b\in\mathbb{C}[\mathcal{S}_N]$ satisfy $\langle au, v\rangle = \langle u, bv\rangle$ for all $u,v\in\mathcal{H}$ iff $b = \overline{a}^{\ast} + c$ with $c\in\mathcal{J}_0$.
\end{Prop}
\begin{proof}
The relations $\langle au, v\rangle = \langle u, \overline{a}^{\ast}v\rangle$ and $\langle au, v\rangle = \langle u, bv\rangle$ for all $u,v\in\mathcal{H}$ lead to $\langle u, (\overline{a}^{\ast}-b)v\rangle = 0$ for all $u,v\in\mathcal{H}$. From this we obtain $(\overline{a}^{\ast}-b)v = 0$ for all $v\in\mathcal{H}$.
\end{proof}
\begin{Lem}$\,$\vspace*{-0.2cm} \label{lem2.2}
\begin{enumerate}
\item[\rm (i)] If $a\in\mathcal{J}$, then $\overline{a}, a^{\ast}, \overline{a}^{\ast}\in\mathcal{J}$, too.
\item[\rm (ii)] If $a\in\mathcal{J}_0$, then $\overline{a}, a^{\ast}, \overline{a}^{\ast}\in\mathcal{J}_0$, too.
\end{enumerate}
\end{Lem}
\begin{proof}
We show (ii). If $au = 0$ for all $u\in\mathcal{H}$, then $\overline{a}w = \overline{a}\,\overline{u} = \overline{au} = 0$ for all $w = \overline{u}\in\mathcal{H}$, i.e. $\overline{a}\in\mathcal{J}_0$. Further, the condition $au = 0$ for all $u\in\mathcal{H}$ leads to $\langle au, v\rangle = \langle u, \overline{a}^{\ast}v\rangle = 0$ for all $u, v\in\mathcal{H}$. Thus we obtain $\overline{a}^{\ast}v = 0$ for all $v\in\mathcal{H}$, i.e. $\overline{a}^{\ast}\in\mathcal{J}_0$. Finally, a combination of $\overline{a}, \overline{a}^{\ast}\in\mathcal{J}_0$ yields $a^{\ast} = \overline{\overline{a}^{\ast}}\in\mathcal{J}_0$.
\end{proof}

$\,$\\
Now it follows from Proposition \ref{prop2.1} and Lemma \ref{lem2.2}
\begin{Prop} A symmetry operator $a\in\mathcal{J}$ is self-adjoint iff $\overline{a}^{\ast} = a$.
\end{Prop}
\begin{proof}
According to Proposition \ref{prop2.1} $a$ is self-adjoint iff $a = \overline{a}^{\ast} + c$ with $c\in\mathcal{J}_0$. But since $a, \overline{a}^{\ast}\in\mathcal{J}$ we obtain also $c = a - \overline{a}^{\ast}\in\mathcal{J}$ such that $c = 0$.
\end{proof}
\begin{Def}
We say that a symmetry operator $a\in\mathbb{C}[\mathcal{S}_N]$ has property {\rm (S)} if $a$ satisfies the condition\\*[-0.4cm]
\parbox{15cm}{\[
\overline{a}^{\ast} = a\,.
\]}\hfill {\rm (S)}
\end{Def}
Every $a\in\mathbb{C}[\mathcal{S}_N]$ with property (S) is self-adjoint. However, an arbitrary symmetry operator $a= b + \tilde{b}\in\mathbb{C}[\mathcal{S}_N]$ with $b\in\mathcal{J}$, $\tilde{b}\in\mathcal{J}_0$ has only to fulfil $\overline{b}^{\ast} = b$ to be a self-adjoint operator. We investigate in our paper operators $a\in\mathbb{C}[\mathcal{S}_N]$ which have property (S).

Now we search for classes of symmetry operators which possess property (S).
\begin{Rem}
Obviously, Young symmetrizers $y_t\in\mathbb{C}[\mathcal{S}_N]$ do not have property {\rm (S)} in general. For instance, one can easily check that the Young symmetrizer $y_t = [1,2,3] + [2,1,3] - [3,1,2] - [3,2,1]\in\mathbb{C}[\mathcal{S}_3]$ of the standard tableau $t = \begin{array}{cc} {\scriptstyle 1} & {\scriptstyle 2} \\ {\scriptstyle 3} & \end{array}$ does not fulfil {\rm (S)}.
\end{Rem}
The following theorem yields a big manifold of symmetry operators with property (S).
\begin{Thm}
Let $G\subseteq\mathbb{C}[\mathcal{S}_N]$ be a subgroup and $\chi$ be an irreducible character of $G$. Then\footnote{$|G|$ denotes the cardinality of $G$.}
\begin{equation}
\chi := \frac{\chi(\mathrm{id})}{|G|}\,\sum_{p\in G} \chi(p)\,p \label{eq5}
\end{equation}
is an idempotent which satisfies {\rm (S)}, too.
\end{Thm}
\begin{proof}
The assertion follows from \cite[Prop.~II.1.47]{fie16}, $\chi(p^ {-1}) = \overline{\chi(p)}$ and $\chi(\mathrm{id})\in\mathbb{R}$.
\end{proof}
\begin{Cor}
If $(G, \epsilon)$ is a commutation symmetry, then $\epsilon : G \rightarrow \mathbb{C}$ is a {\rm 1}-dimensional character of $G$ which is irreducible. Consequently, the idempotent $\epsilon := \frac{1}{|G|}\sum_{p\in G} \epsilon(p)\,p$ of every commutation symmetry has property {\rm (S)}.
\end{Cor}
\begin{proof}
The assertion follows from \cite[Prop.~II.1.48]{fie16}.
\end{proof}

$\,$\\
Already the commutation symmetries lead to a big manifold of symmetry operators with property (S). Let us consider the $\mathcal{S}_6$. A list of all commutation symmetries belonging to subgroups of $\mathcal{S}_N$ with $N\le 6$ is given in \cite[Appendix A.1]{fie16}. The $\mathcal{S}_6$ possesses 55 conjugacy classes of subgroups $G\not=\{\mathrm{id}\}$ (solvable and non-solvable). The number of solvable subgroups $G\not=\{\mathrm{id}\}$ of $\mathcal{S}_6$ is equal to 1428 (see \cite[III.2.2]{fie16}). On all theses subgroups one can define between 1 and 8 commutation symmetries. This yields a big number of operators with property (S). However, if we use irreducible character for the construction of operators with property (S), the number of such operators becomes still bigger.

\begin{Thm} \label{thm2.8}
Let $\mathcal{R}\subset\mathbb{C}[\mathcal{S}_N]$ be a minimal right ideal with $\mathcal{R}\not\subseteq\mathcal{J}_0$. {\rm(} $\mathcal{R}$ defines a symmetry class{\rm)}. Then the set of {\rm (}primitive{\rm )} generating idempotents of $\mathcal{R}$ contains one and only one idempotent $f$ with property {\rm (S)}. $f$ can be formed from an arbitrary generating idempotent $e$ of $\mathcal{R}$ by
\begin{equation}
f = \mu\,e\cdot \overline{e}^{\ast}\,,\;\;\;\mu\in\mathbb{R}\,. \label{eq6}
\end{equation}
\end{Thm}
\begin{Rem}
A corresponding statement holds true for minimal left ideals $\mathcal{L}\subset\mathbb{C}[\mathcal{S}_N]$. For left ideals one has to replace {\rm (\ref{eq6})} by $f = \mu\,\overline{e}^{\ast}\cdot e$.
\end{Rem}
\begin{Rem}
A statement similar to Theorem {\rm\ref{thm2.8}} was proven by H. Weyl in {\rm\cite[p.~295]{weyl4}}. We give here a new proof.
\end{Rem}
\begin{proof}
Since $\mathcal{R}$ is minimal and $\mathcal{R}\not\subseteq\mathcal{J}_0$ we have $\mathcal{R}\subseteq\mathcal{J}$. Let $e$ be a (primitive) generating idempotent of $\mathcal{R}$. Obviously, $h := e\cdot\overline{e}^{\ast}$ is a group ring element with property (S). (Note that $(a\cdot b)^{\ast} = b^{\ast}\cdot a^{\ast}$.) We show that $h$ is essentially idempotent.

Since $e\in\mathcal{J}$ there exists an $u\in\mathcal{H}$ such that $eu\not= 0$. Using the scalar product we obtain $\langle eu , eu\rangle\not= 0$ and $\langle u , (\overline{e}^{\ast}\cdot e)u\rangle\not= 0$. Consequently, we have $(\overline{e}^{\ast}\cdot e)u\not= 0$ and $\overline{e}^{\ast}\cdot e\not= 0$. The relation $e\cdot \overline{e}^{\ast}\not= 0$ can be proved when we start a similar consideration from $\langle \overline{e}^{\ast}v , \overline{e}^{\ast}v\rangle$, where $v\in\mathcal{H}$ is an element with $\overline{e}^{\ast}v\not= 0$.

Now we consider $0\not= \langle (\overline{e}^{\ast}\cdot e)u , (\overline{e}^{\ast}\cdot e)u\rangle = \langle eu , (e\cdot\overline{e}^{\ast}\cdot e)u\rangle$. This leads to $e\cdot\overline{e}^{\ast}\cdot e\not= 0$. If $e$ is a primitive idempotent then $e\cdot x\cdot e$ is proportional to $e$ for all $x\in\mathbb{C}[\mathcal{S}_N]$. Since $e\cdot\overline{e}^{\ast}\cdot e\not= 0$ we obtain $e\cdot\overline{e}^{\ast}\cdot e = \alpha\,e$ with $\alpha\not= 0$ and $h\cdot h = e\cdot\overline{e}^{\ast}\cdot e\cdot\overline{e}^{\ast} = \alpha (e\cdot\overline{e}^{\ast}) = \alpha\,h$. Thus $h$ ist essentially idempotent.

Next we show that $\alpha\in\mathbb{R}$. From $e\cdot\overline{e}^{\ast}\cdot e = \alpha\,e$ it follows $\overline{e}^{\ast}\cdot e\cdot\overline{e}^{\ast} = \overline{\alpha}\,\overline{e}^{\ast}$. If we multiply this relation from the left by $e$ we obtain $\alpha (\overline{e}^{\ast}\cdot e) = \overline{\alpha} (\overline{e}^{\ast}\cdot e)$ and $\alpha = \overline{\alpha}$. So $f:= \frac{1}{\alpha} h$ is an idempotent which has property (S) because $\alpha\in\mathbb{R}$. Finally $\mathcal{R}' := f\cdot\mathbb{C}[\mathcal{S}_N] = (e\cdot\overline{e}^{\ast})\cdot\mathbb{C}[\mathcal{S}_N]$ is a right ideal with $\mathcal{R}'\subseteq\mathcal{R}$. Since $\mathcal{R}'\ni e\cdot\overline{e}^{\ast}\not=0$ and $\mathcal{R}$ is minimal we obtain $\mathcal{R}' = \mathcal{R}$, i.e. $f$ generates $\mathcal{R}$.

Now we show the uniqueness of $f$. Assume $\mathcal{R}$ possesses two generating idempotents $f_1$, $f_2$ with property (S). We denote by $\mathcal{H}_{\mathcal{R}}$ the symmetry class defined be $\mathcal{R}$ and by $\mathcal{H}_{\mathcal{R}}^{\perp}$ the orthogonal complement of $\mathcal{H}_{\mathcal{R}}$ with respect to $\langle\cdot , \cdot\rangle$. Every idempotent $f_i$ satisfies $f_i x = x$ for all $x\in \mathcal{H}_{\mathcal{R}}$. Since every $x\in\mathcal{H}_{\mathcal{R}}$ has a structure $x = f_i z$ with $z\in\mathcal{H}$ we obtain
$0 = \langle y , f_i z \rangle = \langle \overline{f_i}^{\ast}y,z\rangle = \langle f_i y,z\rangle$ for all $y\in\mathcal{H}_{\mathcal{R}}^{\perp}$ and all $z\in\mathcal{H}$. This leads to $f_i y = 0$ for all $y\in\mathcal{H}_{\mathcal{R}}^{\perp}$. We see that both idempotents $f_1$ and $f_2$ have the same effect on $\mathcal{H}_{\mathcal{R}}$ and $\mathcal{H}_{\mathcal{R}}^{\perp}$. Consequently, we have $f_1 z = f_2 z$ for all $z\in\mathcal{H}$, from which follows $f_1 - f_2\in\mathcal{J}_0$. But since $\mathcal{R}\subseteq\mathcal{J}$, $f_1$ and $f_2$ lie in $\mathcal{J}$ and $f_1 - f_2 = 0$.
\end{proof}

\section{Decomposition of self-adjoint idempotents} \label{sec4}
In this section we present an algorithm for the decomposition of an idempotent $e$ with property (S) into pairwise orthogonal, primitive idempotents $f_i$ which have property (S), too.
\begin{Thm} \label{thm3.1}
Every idempotent $e\in\mathcal{J}$ with $\overline{e}^{\ast} = e$ has a decomposition
\begin{equation}
e = f_1 + f_2 +\ldots + f_l \label{eq7}
\end{equation}
into primitive idempotents $f_1,\ldots f_l$ which fulfill
\begin{equation}
\overline{f_i}^{\ast} = f_i\;\;\;\;and\;\;\;\;f_i\cdot f_j = 0\,,\;\;i\not=j\,.
\end{equation}
\end{Thm}
\begin{proof}
The idempotent $e$ generates a right ideal $\mathcal{R} = e\cdot\mathbb{C}[\mathcal{S}_N]\subseteq\mathcal{J}$.
The decomposition (\ref{eq7}) can be constructed by the following
\begin{Alg}$\;$\\*[-0.5cm] \label{alg3.2}
\begin{enumerate}
\item[\rm (i)] Form a set $\mathcal{Y}$ of primitive idempotents $y\in\mathbb{C}[\mathcal{S}_N]$ such that $\mathbb{C}[\mathcal{S}_N] = \bigoplus_{y\in\mathcal{Y}} y\cdot\mathbb{C}[\mathcal{S}_N]$.
\item[\rm (ii)] Determine a $y_1\in\mathcal{Y}$ such that $e\cdot y_1\not= 0$. The group ring element $e\cdot y_1$ generates a minimal right ideal $\mathcal{R}_1 := e\cdot y_1\cdot\mathbb{C}[\mathcal{S}_N]$ with $\mathcal{R}_1\subseteq\mathcal{R}$.
\item[\rm (iii)] Determine a primitive generating idempotent $h_1$ of $\mathcal{R}_1$ from $e\cdot y_1$.
\item[\rm (iv)] Calculate the unique generating idempotent $f_1 = {\mu}_1\,h_1\cdot\overline{h_1}^{\ast}$ of $\mathcal{R}_1$.
\item[\rm (v)] Calculate the rest $r_1 := e - f_1$.
\item[\rm (vi)] Determine further idempotents $f_2, f_3,\ldots$ by iteration of the steps {\rm (ii), \ldots, (v)} starting with the rests $r_1, r_2,\ldots$ instead of $e$. Stop this iteration when a rest $r_l = 0$ was reached.
\end{enumerate}
\end{Alg}
The set $\mathcal{Y}$ could be constracted from the Young symmetrizers of all standard tableaux of $\mathbb{C}[\mathcal{S}_N]$. In step (ii) the set $\mathcal{Y}$ is always the original set $\mathcal{Y}$. We do not delete from $\mathcal{Y}$ the elements $y_1, y_2,\ldots$ found in previous steps (ii). Step (iii) can be carried out by a procedure given in \cite[Prop.I.2.1]{fie16}, \cite[Prop.1]{fie14} or \cite[Prop.4.1]{fie18}

Now we show that Algorithm \ref{alg3.2} produces the decomposition (\ref{eq7}). Obviously, the rest $r_1 = e - f_1$ has property (S). Further it follows $e\cdot f_1 = f_1$ from $\mathcal{R}_1\subseteq\mathcal{R}$. This leads to $r_1\cdot f_1 = 0$. Now we can carry out the calculation
\begin{equation}
\forall\,x,y\in\mathcal{H}:\;0 = \langle x,(r_1\cdot f_1)y\rangle = \langle r_1 x,f_1 y\rangle = \langle (f_1\cdot r_1) x,y\rangle\,, \label{eq9}
\end{equation}
which yield $(f_1\cdot r_1) x = 0$ (for all $x\in\mathcal{H}$) and $f_1\cdot r_1 = 0$ since $f_1, r_1\in\mathcal{J}$. From $f_1\cdot r_1 = 0$ we obtain $r_1\cdot r_1 = e\cdot r_1 - f_1\cdot r_1 = r_1 - 0 = r_1$, i.e. $r_1$ is an idempotent with property (S). Thus the start of a second iteration of {\rm (ii), \ldots, (v)} with $r_1$ instead of $e$ is correct. The result of this iteration are idempotents $f_2, r_2$ which fulfil $r_1 = f_2 + r_2$, $r_1\cdot f_2 = f_2$.

Consider the symmetry classes $\mathcal{H}_{\mathcal{R}_1}$, $\mathcal{H}_{\mathcal{R}_2}$, $\mathcal{H}_{\mathcal{R'}_1}$ of the right ideals $\mathcal{R}_1 = f_1\cdot\mathbb{C}[\mathcal{S}_N]$, $\mathcal{R}_2 = f_2\cdot\mathbb{C}[\mathcal{S}_N]$, $\mathcal{R'}_1 = r_1\cdot\mathbb{C}[\mathcal{S}_N]$. We obtain $\mathcal{R}_1\perp\mathcal{R'}_1$ from (\ref{eq9}) and $\mathcal{R}_2\subseteq\mathcal{R'}_1$ from $r_1\cdot f_2 = f_2$. Consequently, it holds $\mathcal{R}_1\perp\mathcal{R}_2$, too. Now a calculation similar to (\ref{eq9}) with $f_2$ instead of $r_1$ yields $f_1\cdot f_2 = f_2\cdot f_1 = 0$.

Since these considerations can be carried out in all iteration steps of Algorithm \ref{alg3.2} Theorem \ref{thm3.1} is correct.
\end{proof}

\section{Use of discrete Fourier transforms}
Computer calculations in big group rings $\mathbb{C}[\mathcal{S}_N]$ have high costs in calculation time and memory because of the high number of permutations in $\mathcal{S}_N$. One has to use a {\it discrete Fourier transform} for $\mathcal{S}_N$ to reduse these costs. (See \cite{clausbaum1} for discrete Fourier transforms of groups.)
\begin{Def}
A discrete Fourier transform for $\mathcal{S}_N$ is an isomorphism
\begin{eqnarray}
D:\;\mathbb{C}[\mathcal{S}_N] \;\rightarrow\; \bigotimes_{\lambda\vdash N} {\mathbb{C}}^{d_{\lambda}\times d_{\lambda}} & , &
D:\;a = \sum_{p\in\mathcal{S}_N} a_p\,p \;\mapsto\; D(a) =
\left(
\begin{array}{cccc}
A_{{\lambda}_1} &0 &0 &0 \\
0 & A_{{\lambda}_2} & 0 & 0 \\
\vdots & \vdots & \ddots & \vdots \\
0 & 0 & 0 & A_{{\lambda}_l} \\
\end{array}
\right)\hspace*{0.5cm}\; \label{eq10}
\end{eqnarray}
according to Wedderburn's theorem which maps the group ring $\mathbb{C}[\mathcal{S}_N]$ onto an outer direct product $\bigotimes_{\lambda\vdash N} {\mathbb{C}}^{d_{\lambda}\times d_{\lambda}}$ of full matrix rings ${\mathbb{C}}^{d_{\lambda}\times d_{\lambda}}$. We denote by $D_{\lambda}$ the natural projection $D_{\lambda}:\mathbb{C}[\mathcal{S}_N]\rightarrow {\mathbb{C}}^{d_{\lambda}\times d_{\lambda}}$, $D_{\lambda}(a) = A_{\lambda}$, belonging to the partition $\lambda\vdash N$ of $N$.
\end{Def}
The dimension $d_{\lambda}$ of the $d_{\lambda}\times d_{\lambda}$-block matrices $A_{\lambda}$ can be computed from $\lambda$ by means of the {\it hook length formula} (see e.g. \cite{boerner,full4,jameskerb,kerber} or \cite[p.38]{fie16}).

The group ring $\mathbb{C}[\mathcal{S}_N]$ decomposes into {\it minimal two-sided ideals} $\mathcal{Z}_{\lambda}$,
\begin{equation}
\mathbb{C}[\mathcal{S}_N] = \bigoplus_{\lambda\vdash N} \mathcal{Z}_{\lambda} = \bigoplus_{\lambda\vdash N} \mathbb{C}[\mathcal{S}_N]\cdot z_{\lambda}\,. \label{eq11}
\end{equation}
Every $\mathcal{Z}_{\lambda}$ is generated by a unique idempotent $z_{\lambda}$ which is {\it centrally primitive}. $z_{\lambda}$ can be calculated by means of (\ref{eq5}) from the irreducible character ${\chi}_{\lambda}$ belonging to $\lambda\vdash N$. Consequently every $z_{\lambda}$ has property (S).

The structure of block matrices in (\ref{eq10}) reflects the decomposition (\ref{eq11}) since $D_{\lambda}: \mathcal{Z}_{\lambda}\rightarrow {\mathbb{C}}^{d_{\lambda}\times d_{\lambda}}$ is an isomorphism.

Elements $a\in \mathbb{C}[\mathcal{S}_N]$ which lie in minimal left, right or two-sided ideals of $\mathbb{C}[\mathcal{S}_N]$, have only one non-vanishing block matrix in (\ref{eq10}). In particular this holds true for all primitive idempotents. The product $a\cdot b$ of such an element $a$ by another element $b$ is in $\mathbb{C}[\mathcal{S}_N]$ the product of two long sums $a = \sum_p a_p p$, $b = \sum_p b_p p$ of length $N!$, however in $\bigotimes_{\lambda\vdash N} {\mathbb{C}}^{d_{\lambda}\times d_{\lambda}}$ only the product of two $d_{\lambda}\times d_{\lambda}$-matrices $A_{\lambda}\cdot B_{\lambda}$ which are much smaller. This leads to a considerable reduction of computation costs.

The steps (i)-(iii) of Algorithm \ref{alg3.2} were taken from another ideal decomposition algorithm which we described in \cite[Chap.I.2]{fie16} or \cite{fie14,fie18}. The algorithm from \cite{fie16,fie14,fie18} runs both in $\mathbb{C}[\mathcal{S}_N]$ and in $\bigotimes_{\lambda\vdash N} {\mathbb{C}}^{d_{\lambda}\times d_{\lambda}}$ without a necessity of a Fourier transformation between these two spaces during the run of the algorithm. The version in $\bigotimes_{\lambda\vdash N} {\mathbb{C}}^{d_{\lambda}\times d_{\lambda}}$ is much more efficient than the version in $\mathbb{C}[\mathcal{S}_N]$.

Because of the relationship between Algorithm \ref{alg3.2} and the algorithm from \cite{fie16,fie14,fie18} one can transfer Algorithm \ref{alg3.2} to $\bigotimes_{\lambda\vdash N} {\mathbb{C}}^{d_{\lambda}\times d_{\lambda}}$ by means of the remarks in \cite[p.46]{fie16}, \cite[Sec.5]{fie14}, \cite[pp.9-10]{fie18}.

Only the computation of $a^{\ast}$ for an element $a\in\mathbb{C}[\mathcal{S}_N]$ can not be transfered to $\bigotimes_{\lambda\vdash N} {\mathbb{C}}^{d_{\lambda}\times d_{\lambda}}$ in a simple way. The calculation of the matrix $B = D(a^{\ast})$ from a matrix $A = D(a)$ can not be carried out in $\bigotimes_{\lambda\vdash N} {\mathbb{C}}^{d_{\lambda}\times d_{\lambda}}$ without Fourier transformations between $\bigotimes_{\lambda\vdash N} {\mathbb{C}}^{d_{\lambda}\times d_{\lambda}}$ and $\mathbb{C}[\mathcal{S}_N]$. In the next sections we present algorithms for the efficient computation of $B$ from $A$.
\begin{Rem}
For $\mathcal{S}_N$ three discrete Fourier transforms are known:
\begin{enumerate}
\item[\rm (a)] Young's natural representation of $\mathcal{S}_N$ {\rm (}see {\rm\cite[p.~51]{fie16})},
\item[\rm (b)] Young's semi-normal representation of $\mathcal{S}_N$ {\rm (}see {\rm\cite[p.~55]{fie16})},
\item[\rm (c)] Young's orthogonal representation of $\mathcal{S}_N$ {\rm (}see {\rm\cite[pp.~133-135]{boerner2})}.
\end{enumerate}
The {\rm DFT} {\rm (b)} is the basis of the fast Fourier transform for $\mathcal{S}_N$ by M. Clausen and U. Baum {\rm\cite{clausbaum1}}. We use {\rm (a)} as discrete Fourier transform in our {\sf Mathematica} package {\sf PERMS} {\rm\cite{fie10}}.

The three {\rm DFT} map permutations $p\in\mathcal{S}_N$ to following types of matrices{\rm :}
\begin{center}
{\rm
\begin{tabular}{c|c|c|c}
 & (a) & (b) & (c) \\
\hline
$D(p)$ & integer matrix & rational matrix & real matrix \\
\end{tabular}
}
\end{center}
If we want to calculate $D(a)$ of an $a\in\mathbb{C}[\mathcal{S}_N]$ whose coefficients are only integers or rational numbers then {\rm (a)} and {\rm (b)} allow calculations not only with a certain numerical precision but with infinite precision. In Section {\rm\ref{sec5}} we will see that an infinite precision has not  an advantage in any case. If a calculation lead to such big integers which a computer program can not completely decompose into primes than a strong growth of integers can arise and cause memory problems.
\end{Rem}

\section{Calculation of $D(a^{\ast})$ from $D(a)$} \label{sec5}
Now we investigate the question how one can calculate the matrix $B = D(a^{\ast})\in\bigotimes_{\lambda\vdash N} {\mathbb{C}}^{d_{\lambda}\times d_{\lambda}}$ from the matrix $A = D(a)\in\bigotimes_{\lambda\vdash N} {\mathbb{C}}^{d_{\lambda}\times d_{\lambda}}$ without a discrete Fourier transformation between $\bigotimes_{\lambda\vdash N} {\mathbb{C}}^{d_{\lambda}\times d_{\lambda}}$ and $\mathbb{C}[\mathcal{S}_N]$.
\begin{Prop} \label{prop5.1}
Every $a\in\mathbb{C}[\mathcal{S}_N]$ has a unique decomposition
\begin{equation} \label{eq12}
a = \sum_{\lambda\vdash N} a_{\lambda}\;\;\;where\;\;\;a_{\lambda} = a\cdot z_{\lambda}\in\mathcal{Z}_{\lambda}\,.
\end{equation}
The $a_{\lambda}$ fulfil
\begin{equation}
a_{\lambda}^{\ast} = a^{\ast}\cdot z_{\lambda}\in\mathcal{Z}_{\lambda}\,.
\end{equation}
\end{Prop}
\begin{proof}
(\ref{eq12}) follows from (\ref{eq11}). Since $z_{\lambda}$ has property (S) and commutes with every element from $\mathbb{C}[\mathcal{S}_N]$ we have furthermore $a_{\lambda}^{\ast} = z_{\lambda}^{\ast}\cdot a^{\ast} = a^{\ast}\cdot z_{\lambda}$.
\end{proof}

Because of Proposition \ref{prop5.1} we can restrict us to
\begin{Prbl} \label{probl5.2}
Determine a matrix-valued function $f:{\mathbb{C}}^{d_{\lambda}\times d_{\lambda}}\rightarrow {\mathbb{C}}^{d_{\lambda}\times d_{\lambda}}$ which fulfils
\begin{equation}
f(A) = B\;\;\;\;\Leftrightarrow\;\;\;\;\exists\,a\in\mathcal{Z}_{\lambda}:\;A = D_{\lambda}(a)\;\wedge\;B = D_{\lambda}(a^{\ast}) \label{eq14}
\end{equation}
\end{Prbl}
Obviously, a function (\ref{eq14}) exists for every $\mathcal{Z}_{\lambda}$. Every such $f$ is a linear function defined on the whole of ${\mathbb{C}}^{d_{\lambda}\times d_{\lambda}}$. If we can determine $f$ we obtain a fast algorithm for the computation of $B$ from $A$ without a Fourier transformation between $\bigotimes_{\lambda\vdash N} {\mathbb{C}}^{d_{\lambda}\times d_{\lambda}}$ and $\mathbb{C}[\mathcal{S}_N]$.

The simplest algorithm which gives a solution of Problem \ref{probl5.2} is
\begin{Alg} Let be given a $(d_{\lambda}\times d_{\lambda})$-matrix $A$ of symbols representing a natural projection $D_{\lambda}(a)$ of an arbitrary element $a\in\mathcal{Z}_{\lambda}$.
\begin{enumerate}
\item[\rm (i)] Calculate $a = (D_{\lambda}|_{\mathcal{Z}_{\lambda}})^{-1}(A)$.
\item[\rm (ii)] Calculate $a^{\ast}$ from $a$.
\item[\rm (iii)] Calculate $B = D_{\lambda}(a^{\ast})$.
\end{enumerate}
$B$ is a $(d_{\lambda}\times d_{\lambda})$-matrix of linear functions of the elements of $A$. The mapping $f:A\mapsto B$ is the solution of Problem {\rm\ref{probl5.2}}. \label{alg5.3}
\end{Alg}
Algorithm \ref{alg5.3} needs a computer algebra system such as {\sf Mathematica} which allows symbolic computations. It can lead to high costs in calculation time because it carries out a Fourier transformations of a relatively small $(d_{\lambda}\times d_{\lambda})$-matrix $A$ to a long group ring element $a$ of length $N!$ and a second Fourier transformation of a group ring element $a^{\ast}$ of length $N!$ to a $(d_{\lambda}\times d_{\lambda})$-matrix $B$. However these costs arise only once when we run Algorithm \ref{alg5.3}. When we have determined the function $f$ then we can very fast compute $B$ from $A$ for all $A\in {\mathbb{C}}^{d_{\lambda}\times d_{\lambda}}$.

The costs in memory for Algorithm \ref{alg5.3} are smaller than for the other algorithms presented here. Using Algorithm \ref{alg5.3} one runs rarely the risk of memory problems.

A second type of algorithms solving Problem \ref{probl5.2} bases on the following consideration. For every $\mathcal{Z}_{\lambda}$ the set $\{ p\cdot z_{\lambda}\;|\;p\in\mathcal{S}_N\}$ is a generating system of the vector space $\mathcal{Z}_{\lambda}$. We can determine a subset $\mathcal{P}_{\lambda}\subset \mathcal{S}_N$ of permutations such that $\mathcal{B}_{\lambda} := \{ p\cdot z_{\lambda}\;|\;p\in\mathcal{P}_{\lambda}\}$ is a basis of $\mathcal{Z}_{\lambda}$. If $\mathcal{P}_{\lambda}$ is known we can write every $a\in\mathcal{Z}_{\lambda}$ in the following form
\begin{equation}
a = a\cdot z_{\lambda} = \sum_{p\in\mathcal{S}_N} a_p\,p\cdot z_{\lambda} = \sum_{p\in\mathcal{P}_{\lambda}} {\alpha}_p\,p\cdot z_{\lambda} = \alpha\cdot z_{\lambda}\;\;\;\mathrm{with}\;\;\;\alpha := \sum_{p\in\mathcal{P}_{\lambda}} {\alpha}_p\,p\,.
\end{equation}
It holds true for $a$, $\alpha$ and $z_{\lambda}$
\begin{equation}
D_{\lambda}(a) = D_{\lambda}(\alpha)\,,\;\;\;\;D_{\lambda}(z_{\lambda}) = \mathrm{Id}\,,\;\;\;\;a^{\ast} = z_{\lambda}^{\ast}\cdot {\alpha}^{\ast} = {\alpha}^{\ast}\cdot z_{\lambda}\,.
\end{equation}
\begin{Def}
\begin{enumerate}
\item[\rm (i)] Let $\alpha = \sum_{p\in\mathcal{P}_{\lambda}} {\alpha}_p\,p$ be a group ring element. We denote by $v[{\alpha}_p]$ the column vector in which the ${\alpha}_p$ were arranged according to the lexicographical order of permutations.
\item[\rm (ii)] Let $A = (A_{ij})$ be a matrix. We denote by $v[A_{ij}]$ the column vector in which the elements $A_{ij}$ were written line by line.
\end{enumerate}
\end{Def}
Now we formulate our second algorithm which solves Problem \ref{probl5.2}.
\begin{Alg}
Let $A = (A_{ij})$ be a symbolic $(d_{\lambda}\times d_{\lambda})$-matrix and $\alpha = \sum_{p\in\mathcal{P}_{\lambda}} {\alpha}_p\,p$ be a group ring element with symbolic coefficients ${\alpha}_p$.
\begin{enumerate}
\item[\rm (i)] Determine the coefficient matrix $\Phi$ of the equation
\begin{equation}
v\left[D_{\lambda}(\alpha)\right] = \Phi\cdot v[{\alpha}_p]\,.
\end{equation}
\item[\rm (ii)] Calculate the inverse matrix ${\Phi}^{-1}$.
\item[\rm (iii)] Determine the coefficient matrix $\Psi$ of the equation
\begin{equation}
v\left[D_{\lambda}({\alpha}^{\ast})\right] = \Psi\cdot v[{\alpha}_p]\,.
\end{equation}
\item[\rm (iv)] Form the function
\begin{equation}
\tilde{f}:\,v[A_{ij}]\mapsto \Psi\cdot {\Phi}^{-1}\cdot v[A_{ij}]\,.
\end{equation}
Arrange the elements of the vectors $v[A_{ij}]$ and $\Psi\cdot {\Phi}^{-1}\cdot v[A_{ij}]$ in $(d_{\lambda}\times d_{\lambda})$-matrices to obtain the function $f: {\mathbb{C}}^{d_{\lambda}\times d_{\lambda}}\rightarrow {\mathbb{C}}^{d_{\lambda}\times d_{\lambda}}$.
\end{enumerate} \label{alg5.5}
\end{Alg}
It is clear that Algorithm \ref{alg5.5} yields the function $f$ searched in Problem \ref{probl5.2}. Note that $\mathcal{P}_{\lambda} = \dim\mathcal{Z}_{\lambda} = d_{\lambda}^2$. So both vector $v[{\alpha}_p]$ and vector $v[A_{ij}]$ has length $d_{\lambda}^2$ and $\Phi$ and $\Psi$ are quadratic $(d_{\lambda}^2\times d_{\lambda}^2)$-matrices. $\Phi$ has an inverse matrix because $\mathcal{B}_{\lambda}$ is a basis of $\mathcal{Z}_{\lambda}$ and $D_{\lambda}|_{\mathcal{Z}_{\lambda}}: \mathcal{Z}_{\lambda}\rightarrow {\mathbb{C}}^{d_{\lambda}\times d_{\lambda}}$ is an isomorphism. The use of $v[{\alpha}_p]$ in the right-hand side is correct because $\alpha$ and ${\alpha}^{\ast}$ have the same coefficients. Only the places of these coefficients are different in $\alpha$ and ${\alpha}^{\ast}$.

By means of Algorithm \ref{alg5.5} we calculated the functions $f$ for all partitions $\lambda\vdash 5$, $\lambda\vdash 6$ and all partitions $\lambda\vdash 7$, $\lambda\not= (4\,2\,1), (3\,2\,1^2)$. The following table shows some information concerning the calculations in $\mathcal{S}_7$.
\begin{center}
\begin{tabular}{l|r|r|r|r|r|r}
$\lambda$ & $d_{\lambda}$ & $d_{\lambda}^2$ & $d_{\lambda}^4$ & time & RAM & f\\
\hline
$(6\,1)$ & 6      & 36  & 1296   & 0.6s   &         & 13KB\\
$(5\,2)$ & 14     & 196 & 38416  & 28.6s  &         & 506KB\\
$(5\,1^2)$ & 15   & 225 & 50625  & 31.2s  &         & 182KB\\
$(4\,3)$ & 14     & 196 & 38416  & 35.9s  &         & 494KB\\
$(4\,1^3)$ & 20   & 400 & 160000 & 278.8s & 131.7MB & 404KB\\
$(3^3\,1)$ & 21   & 441 & 194481 & 615.5s & 166.9MB & 2149KB\\
$(3\,2^2)$ & 21   & 441 & 194481 & 624.3s & 166.9MB & 2187KB\\
$(3\,1^4)$ & 15   & 225 & 50625  & 31.3s  & 34.6MB  & 182KB\\
$(2^3\,1)$ & 14   & 196 & 38416  & 36.1s  & 34.8MB  & 471KB\\
$(2^2\,1^3)$ & 14 & 196 & 38416  & 30.7s  & 33.6MB  & 433KB\\
$(2\,1^5)$ & 6    & 36  & 1296   & 0.6s   & 14.7MB  & 13KB\\
\end{tabular}
\end{center}
The colums of $d_{\lambda}^2$ and $d_{\lambda}^4$ shows the number of elements in $v[{\alpha}_p]$, $v[A_{ij}]$ and $\Phi$, $\Psi$, respectively. 'time' and 'RAM' give the computation time and the memory, respectively, which the algorithm needed. 'f' shows the length of the file in which the function $f$ was stored.

For the partitions $\lambda = (4\,2\,1), (3\,2\,1^2)$ a big memory problem arose from the use of exact integers in the calculation. It turned out, that for instance the matrix $\Phi$ for $\lambda = (4\,2\,1)$ has the determinant
\begin{center}
{\scriptsize
\begin{tabular}{r}
-96974257961965049074024912099372004573795424708044922073363953457157880990336620245014730563217365\\
    946027469537574499367764896357932972629577610678133530230415476788046519825483493859992633559136\\ 
   303784510251434033116082118775724498148053570956977145676159650656795201937451001664676330006575\\
   437913887075107664190746396865109859747118150041195008145445396420723507449624828927005702309095\\ 
   274295858496179822968799982444105484313810897703096007568025010485731184828349379597074706284167\\ 
   453662581507561255480875750071503321884946609362947141428351638033559982301650092032000000000000\\ 
   000000000000000\\
(593 digits)\\
\end{tabular}
}
\end{center}
During the calculation of ${\Phi}^{-1}$ rational numbers are generated whose numerators and denominators are such large integers as the above determinat. The computer algebra system is no longer able to find all prime factors of these integers and to reduce the rational numbers. So the integers become longer and longer and the available memory is exceeded.

We used Algorithm \ref{alg5.3} for the determination of $f$ for $\lambda = (4\,2\,1), (3\,2\,1^2)$. The calculation had the following characteristics:
\begin{center}
\begin{tabular}{l|r|r|r|r|r|r}
$\lambda$ & $d_{\lambda}$ & $d_{\lambda}^2$ & $d_{\lambda}^4$ & time & RAM & f\\
\hline
$(4\,2\,1)$   & 35 & 1225 & 1500625 & 4.393h & 440.8MB & 13974KB \\
$(3\,2\,1^2)$ & 35 & 1225 & 1500625 & \\
\end{tabular}
\end{center}

\section{Determination of the sets $\mathcal{P}_{\lambda}$}
In this section we present an algorithm which determines the sets $\mathcal{P}_{\lambda}$ for the Algorithm \ref{alg5.5} (see Figure 1).
\begin{figure}
\begin{center}
\fbox{
\parbox{12cm}{
\begin{tabbing}
{\bf output:} \=x \kill \\
{\bf input:} \>$\lambda\vdash N$\\*[0.2cm]
{\bf output:} \>$\mathcal{P}_{\lambda}$\\
\end{tabbing}
\begin{tabbing}
{\bf beg}\={\bf while} \={\bf then} \=x\kill
{\bf begin} \\
\>$\xi := (x_1, x_2,\ldots, x_{d_{\lambda}^2})$;\\
\>$p := \mathrm{id}$;\\
\>$\mathcal{P}_{\lambda} := \{\mathrm{id}\}$;\\
\>$\tau(\xi) := \xi\cdot v[D_{\lambda}(p)]$;\\
\>{\bf solve} $\tau(\xi) = 0\;\;\Rightarrow\;\;${\bf result}: $x_s = \sigma($rest of the $x_i)$;\\
\>$\xi := \xi \leftarrow$ {\bf substitute} $x_s = \sigma($rest of the $x_i)$;\\
\>{\bf while} $|\mathcal{P}_{\lambda}| < d_{\lambda}^2$ {\bf do}\\
\> \>$p :=$ {\bf next-permutation}$(p)$;\\
\>\>$\tau(\xi) := \xi\cdot v[D_{\lambda}(p)]$;\\
\>\>{\bf if} \>$\tau(\xi)\not\equiv 0$ \\
\>\>{\bf then} \>{\bf solve} $\tau(\xi) = 0\;\;\Rightarrow\;\;${\bf result}: $x_s = \sigma($rest of the $x_i)$;\\
\>\>\>$\xi := \xi \leftarrow$ {\bf substitute} $x_s = \sigma($rest of the $x_i)$;\\
\>\>\>$\mathcal{P}_{\lambda} := \mathcal{P}_{\lambda}\cup\{p\}$;\\
\>\>{\bf fi};\\
\>{\bf od};\\
\>{\bf return} $\mathcal{P}_{\lambda}$;\\
{\bf end};
\end{tabbing}
}
}
\end{center}
\caption{Algorithm for determining of $\mathcal{P}_{\lambda}$}
\end{figure}

We can not simply consider the set $\{p\cdot z_{\lambda}\;|\;p\in\mathcal{S}_N\}$ und determine $\dim\mathcal{Z}_{\lambda} = d_{\lambda}^2$ linearly independent vectors in this set, because this set contains $N!$ vectors from $\mathbb{C}[\mathcal{S}_N]$ of length $N!$. This data set is very large.

The algorithm calculates
\begin{equation}
D_{\lambda}(p\cdot z_{\lambda}) = D_{\lambda}(p)
\end{equation}
for permutations $p\in\mathcal{S}_N$ and checks whether these matrices are linearly independent.

The algorithm starts with the identity permutation $p = \mathrm{id}\in\mathcal{S}_N$ and $\mathcal{P}_{\lambda} = \{\mathrm{id}\}$. To characterize the 1-dimensional linear space spanned by $D_{\lambda}(\mathrm{id})$, the algorithm uses a vector $\xi = (x_1,\ldots , x_{d_{\lambda}^2})$ of variables $x_i$.

The algorithm forms the scalar product $\tau(\xi) = \xi\cdot v[D_{\lambda}(\mathrm{id})]$ of the vectors $\tau(\xi)$ and $[D_{\lambda}(\mathrm{id})]$. $\tau(\xi)$ is a linear function. The algorithm determines one of the variables $x_i$ from the equation $\tau(\xi) = 0$ and obtains a relation $x_s = \sigma($rest of the $x_i)$. By means of this relation the variable $x_s$ is eliminated in $\xi$. We obtain a vector $\xi$ which fulfil $\xi\cdot w = 0$ exactly for the elements $w$ of the 1-dimensional span of $v[D_{\lambda}(\mathrm{id})]$.

Now the algorithm investigates the next permutation $p$. To select $p$, the algorithm uses a procedure {\bf next-permutation} which determines for a given permutation $q\in\mathcal{S}_N$ the successor $p\in\mathcal{S}_N$ according to the lexicographical order of permutations. Then the algorithm checks the condition
\begin{equation}
\tau(\xi) := \xi\cdot v[D_{\lambda}(p)] = 0\,. \label{eq21}
\end{equation}
If (\ref{eq21}) is satisfied, then $p$ is canceled. If, however, (\ref{eq21}) is not fulfilled then a new linear equation $\tau(\xi) = 0$ arises which the algorithm uses to eliminate a further variable in $\xi$. $\xi$ caracterizes now the span of two linearly independent vectors $v[D_{\lambda}(\mathrm{id})]$, $v[D_{\lambda}(p)]$. The new found permutation $p$ ist added to the set $\mathcal{P}_{\lambda}$. This procedure is iterated until the cardinality of $\mathcal{P}_{\lambda}$ has reached the value $d_{\lambda}^2$.

The algorithm needs not much memory because only one matrix $D_{\lambda}(p)$ is stored in the computer memory in every step of the algorithm.
A halving of the calculation time can be obtained by the following theorem.
\begin{Thm}
Let a set $\mathcal{P}_{\lambda}$ be known for a partition $\lambda\vdash N$. Then $\mathcal{B}_{\lambda'} := \{p\cdot z_{\lambda'}\;|\;p\in\mathcal{P}_{\lambda}\}$ is a basis of $\mathcal{Z}_{\lambda'}$, where $\lambda' \vdash N$ denotes that partition whose Young frame is transposed to the Young frame of $\lambda$. \label{thm6.1}
\end{Thm}
\begin{proof}
The generating idempotents $z_{\lambda}$, $z_{\lambda'}$ of $\mathcal{Z}_{\lambda}$, $\mathcal{Z}_{\lambda'}$ fulfil
\begin{equation}
z_{\lambda} := \frac{{\chi}_{\lambda}(\mathrm{id})}{N!}\,\sum_{p\in\mathcal{S}_N} {\chi}_{\lambda}(p)\,p\,,\;\;\;\;
z_{\lambda'} := \frac{{\chi}_{\lambda'}(\mathrm{id})}{N!}\,\sum_{p\in\mathcal{S}_N} {\chi}_{\lambda'}(p)\,p\,,
\end{equation}
where ${\chi}_{\lambda}$, ${\chi}_{\lambda'}$ denote the irreducible characters of $\mathcal{S}_n$ belonging to $\lambda$, $\lambda'$. It holds true ${\chi}_{\lambda}(\mathrm{id})= {\chi}_{\lambda'}(\mathrm{id})$, because ${\chi}_{\lambda}(\mathrm{id})= \sqrt{\dim\mathcal{Z}_{\lambda}}$, ${\chi}_{\lambda'}(\mathrm{id})= \sqrt{\dim\mathcal{Z}_{\lambda'}}$ and $\dim\mathcal{Z}_{\lambda}=\dim\mathcal{Z}_{\lambda'}$. Further, it follows from \cite[Vol.~246, p.~65, Statement 4.11]{kerber}, that
\begin{equation}
\forall\;p\in\mathcal{S}_N:\;{\chi}_{\lambda'}(p) = \mathrm{sign}(p)\cdot {\chi}_{\lambda}(p)\,.
\end{equation}
Because of these facts we can write
\begin{equation}
z_{\lambda} = g_{\lambda} + u_{\lambda}\,\;\;\;\; z_{\lambda'} = g_{\lambda} - u_{\lambda}\,,
\end{equation}
where
\begin{equation}
g_{\lambda} := \frac{{\chi}_{\lambda}(\mathrm{id})}{N!}\,\sum_{p\;\mathrm{even}} {\chi}_{\lambda}(p)\,p\,,\;\;\;\;
u_{\lambda} := \frac{{\chi}_{\lambda}(\mathrm{id})}{N!}\,\sum_{p\;\mathrm{odd}} {\chi}_{\lambda}(p)\,p\,.
\end{equation}

Now we assume that $\mathcal{B}_{\lambda} = \{p\cdot z_{\lambda}\;|\;p\in\mathcal{P}_{\lambda}\}$ is a basis of $\mathcal{Z}_{\lambda}$, but $\mathcal{B}_{\lambda'} := \{p\cdot z_{\lambda'}\;|\;p\in\mathcal{P}_{\lambda}\}$ is a set of linearly dependent vectors from $\mathcal{Z}_{\lambda'}$. Consequently, we can form a non-trivial vanishing linear combination of the vectors from $\mathcal{B}_{\lambda'}$,
\begin{equation}
\sum_{p\in\mathcal{P}_{\lambda}} {\alpha}_p\,p\cdot z_{\lambda'} = 0\,.
\end{equation}
Now the element $\alpha := \sum_{p\in\mathcal{P}_{\lambda}} {\alpha}_p\,p$ has a decomposition $\alpha = \gamma + \omega$, where $\gamma$ and $\omega$ are sums which run only through the even or odd permutations of $\mathcal{P}_{\lambda}$, respectively. Now we obtain
\begin{equation}
0 = \alpha\cdot z_{\lambda'} = (\gamma + \omega)\cdot (g_{\lambda} - u_{\lambda}) = \underbrace{(\gamma\cdot g_{\lambda} - \omega\cdot u_{\lambda})}_{\hspace*{15pt}=: \Gamma} + \underbrace{(\omega\cdot g_{\lambda} - \gamma\cdot u_{\lambda})}_{\hspace*{15pt}=: \Omega}\,. \label{eq27}
\end{equation}
It follows from (\ref{eq27}) that
\begin{equation}
\Gamma = 0\,,\;\;\;\;\Omega = 0\,,
\end{equation}
since $\Gamma$ and $\Omega$ contain only even or odd permutations, respectively.

Now we consider the non-trivial linear combination $\beta := \gamma - \omega$ of the permutations from $\mathcal{P}_{\lambda}$. The product of $\beta$ and $z_{\lambda}$ yields
\begin{equation}
\beta\cdot z_{\lambda} = (\gamma - \omega)\cdot (g_{\lambda} + u_{\lambda}) = (\gamma\cdot g_{\lambda} - \omega\cdot u_{\lambda}) + (\gamma\cdot u_{\lambda}-\omega\cdot g_{\lambda}) = \Gamma - \Omega = 0\,.
\end{equation}
So, $\beta\cdot z_{\lambda}$ is a non-trivial vanishing linear combination of the vectors from $\mathcal{B}_{\lambda}$, i.e. $\mathcal{B}_{\lambda}$ would not be a basis of $\mathcal{Z}_{\lambda}$ in contradiction to the assumptions of Theorem \ref{thm6.1}
\end{proof}

\section{An example}
We finish with an example. We consider the Young symmetrizer $y_t\in\mathbb{C}[\mathcal{S}_5]$ of the Young tableau $t = \begin{array}{cccc} {\scriptstyle 5} & {\scriptstyle 4} & {\scriptstyle 2} & {\scriptstyle 1} \\ {\scriptstyle 3} \\ \end{array}$, which belongs to the partition $\lambda = (4\,1)$. $y_t$ has the length 48. It is proportional to a primitive idempotent $e = \frac{1}{30}\, y_t$ and generates a 4-dimensional minimal right ideal of $\mathbb{C}[\mathcal{S}_5]$.
Now our algorithms yield
\begin{eqnarray*}
{YS} := D_{\lambda}(y_t) & = & \left(
\begin{array}{rrrr}
0 & 0 & 0 & 0 \\
0 & 0 & 0 & 0 \\
-30 & 0 & 30 & 0 \\
0 & 0 & 0 & 0 \\
\end{array}
\right) \\
{YS}^{\ast} := D_{\lambda}(y_t^{\ast}) & = & \left(
\begin{array}{rrrr}
6 & 6 & -24 & 6 \\
0 & 0 & 0 & 0 \\
-6 & -6 & 24 & -6 \\
0 & 0 & 0 & 0 \\
\end{array}
\right) \\
F := \frac{1}{1440}\;{YS}\cdot {YS}^{\ast} & = & \left(
\begin{array}{rrrr}
0 & 0 & 0 & 0 \\
0 & 0 & 0 & 0 \\
-\frac{1}{4} & -\frac{1}{4} & 1 & -\frac{1}{4} \\
0 & 0 & 0 & 0 \\
\end{array}
\right) \\
\end{eqnarray*}
$D_{\lambda}$ was calculated by means of Young's natural representation of $\mathcal{S}_5$. $F$ represents the primitive idempotent $f$ with property (S) which generates the same right ideal as $y_t$ (see Theorem \ref{thm2.8}). $f$ has the length 120.

One can easily check the properties $YS\cdot YS = 30 YS$, $YS^{\ast}\cdot YS^{\ast} = 30 YS^{\ast}$, $F\cdot F = F$, $F\cdot YS = YS$, $\frac{1}{30}\, YS\cdot F = F$. The relation $F^{\ast} = F$ can be verified by means of our {\sf Mathematica} package {\sf PERMS} \cite{fie10}.

We implemented all algorithms presented in this paper in {\sf PERMS}. {\sf Mathematica} notebooks of the calculation for this paper can be downloaded from \cite{fie21}.

\section*{References}

\begin{thebibliography}{10}

\bibitem{fie07c}
B.~Fiedler.
\newblock Symmetry classes connected with the magnetic {H}eisenberg ring.
\newblock In Tadeusz Lulek, Andrzej Wal, and Barbara Lulek, editors, {\em
  Symmetry and Structural Properties of Condensed Matter (SSPCM 2007)}, volume
  104 of {\em Journal of Physics: Conference Series}, Bristol, Philadelphia,
  2008. IOP Publishing.
\newblock Paper: 012035. Pages 9. Online at jpconf.iop.org.

\bibitem{weyl4}
Hermann Weyl.
\newblock {\em The Theory of Groups and Quantum Mechanics}.
\newblock Number 0-486-60269-9 in Dover Books on Mathematics. Dover
  Publications, Inc., New York, second edition.

\bibitem{KarMu}
M.~Karbach and G.~M{\"u}ller.
\newblock Introduction to the {B}ethe ansatz {I}.
\newblock {\em Computers in Physics}, 11:36--43, 1997.
\newblock Online cond-mat/9809162.

\bibitem{KarHuMu}
M.~Karbach, K.~Hu, and G.~M{\"u}ller.
\newblock Introduction to the {B}ethe ansatz {II}.
\newblock {\em Computers in Physics}, 12:565--573, 1998.
\newblock Online cond-mat/9809163.

\bibitem{KarHuMu2}
M.~Karbach, K.~Hu, and G.~M{\"u}ller.
\newblock Introduction to the {B}ethe ansatz {III}.
\newblock Online cond-mat/0008018, 2000.

\bibitem{jak2lu2c}
P.~Jakubczyk, T.~Lulek, D.~Jakubczyk, and B.~Lulek.
\newblock Construction of {K}ostka matrix at the level of bases.
\newblock In Tadeusz Lulek, Andrzej Wal, and Barbara Lulek, editors, {\em
  Symmetry and Structural Properties of Condensed Matter (SSPCM 2007)},
  volume~30 of {\em Journal of Physics: Conference Series}, Bristol,
  Philadelphia, 2008. IOP Publishing.
\newblock Paper: 012039. Online at jpconf.iop.org.

\bibitem{boerner2}
H.~Boerner.
\newblock {\em Representations of Groups}.
\newblock North-Holland Publishing Company, Amsterdam, 2. revised edition,
  1970.

\bibitem{weyl1}
H.~Weyl.
\newblock {\em The Classical Groups, their Invariants and Representations}.
\newblock Princeton University Press, Princeton, New Jersey, 1939.

\bibitem{fie18}
Bernd Fiedler.
\newblock Ideal decompositions and computation of tensor normal forms.
\newblock In {\em S\'eminaire Lotharingien de Combinatoire}, 2001.
\newblock Electronically published: \verb|http://www.mat.univie.ac.at/~slc|.
  B45g, 16 pp. Archive: \verb|http://arXiv.org/abs/math.CO/0211156|.

\bibitem{jak2lu2}
P.~Jakubczyk, T.~Lulek, D.~Jakubczyk, and B.~Lulek.
\newblock The duality of {W}eyl and linear extension of {K}ostka matrices.
\newblock In Tadeusz Lulek, Andrzej Wal, and Barbara Lulek, editors, {\em
  Symmetry and Structural Properties of Condensed Matter (SSPCM 2005)},
  volume~30 of {\em Journal of Physics: Conference Series}, pages 203--208,
  Bristol, Philadelphia, 2006. IOP Publishing.
\newblock Online at jpconf.iop.org.

\bibitem{jak2lu2b}
D.~Jakubczyk, T.~Lulek, P.~Jakubczyk, and B.~Lulek.
\newblock Geometry and rigged strings in {B}ethe {A}nsatz.
\newblock In Tadeusz Lulek, Andrzej Wal, and Barbara Lulek, editors, {\em
  Symmetry and Structural Properties of Condensed Matter (SSPCM 2005)},
  volume~30 of {\em Journal of Physics: Conference Series}, pages 188--196,
  Bristol, Philadelphia, 2006. IOP Publishing.
\newblock Online at jpconf.iop.org.

\bibitem{fie16}
Bernd Fiedler.
\newblock {\em An Algorithm for the Decomposition of Ideals of Semi-Simple
  Rings and its Application to Symbolic Tensor Calculations by Computer}.
\newblock Habilitationsschrift, Universit{\"a}t Leipzig, Fakult{\"a}t f{\"u}r
  Mathematik und Informatik, Leipzig, Germany, November 1999.
\newblock Online: \verb|http://www.fiemath.de/publicat.htm|.

\bibitem{fie14}
Bernd Fiedler.
\newblock An algorithm for the decomposition of ideals of the group ring of a
  symmetric group.
\newblock In Adalbert Kerber, editor, {\em Actes $39^e$ S\'eminaire
  Lotharingien de Combinatoire, Thurnau, 1997}, Publ. I.R.M.A. Strasbourg.
  Institut de Recherche Math\'ematique Avanc\'ee, Universit\'e Louis Pasteur et
  C.N.R.S. ({URA} 01), 1998.
\newblock Electronically published: \verb|http://www.mat.univie.ac.at/~slc|.
  B39e, 26 pp.

\bibitem{clausbaum1}
Michael Clausen and Ulrich Baum.
\newblock {\em Fast {F}ourier {T}ransforms}.
\newblock BI Wissenschaftsverlag, Mannheim, Leipzig, Wien, Z\"urich, 1993.

\bibitem{boerner}
H.~Boerner.
\newblock {\em Darstellungen von Gruppen}, volume~74 of {\em Die Grundlehren
  der mathematischen Wissenschaften in Einzeldarstellungen}.
\newblock Springer-Verlag, Berlin, G\"{o}ttingen, Heidelberg, 1955.

\bibitem{full4}
S.~A. Fulling, R.~C. King, B.~G. Wybourne, and C.~J. Cummins.
\newblock Normal forms for tensor polynomials: I. {T}he {R}iemann tensor.
\newblock {\em Class. Quantum Grav.}, 9:1151 -- 1197, 1992.

\bibitem{jameskerb}
Gordon~D. James and Adalbert Kerber.
\newblock {\em The Representation Theory of the Symmetric Group}, volume~16 of
  {\em Encyclopedia of Mathematics and its Applications}.
\newblock Addison-Wesley Publishing Company, Reading, Mass., London, Amsterdam,
  Don Mills, Ont., Sidney, Tokyo, 1981.

\bibitem{kerber}
A.~Kerber.
\newblock {\em Representations of Permutation Groups}, volume 240, 495 of {\em
  Lecture Notes in Mathematics}.
\newblock Springer-Verlag, Berlin, Heidelberg, New York, 1971, 1975.

\bibitem{fie10}
Bernd Fiedler.
\newblock {\em {PERMS 2.1 (15.1.1999)}}.
\newblock Mathematisches Institut, Universit\"at Leipzig, Leipzig, 1999.
\newblock Will be sent in to MathSource, Wolfram Research Inc.

\bibitem{fie21}
Bernd Fiedler.
\newblock Examples of calculations by means of {PERMS}. {M}athematica
  notebooks.
\newblock Internet \verb|http://www.fiemath.de/pnbks.htm|.

\end{thebibliography}

\end{document}